\documentclass[oneside, 10pt]{amsart}
\usepackage{latexsym, bbm, enumerate, amssymb, amsmath, bm}
\usepackage[full,dcucite]{harvard} 
\usepackage[dvips]{graphicx}
\usepackage{setspace}
\usepackage{paralist}
\usepackage{ifpdf}

\usepackage[show]{ed}

\usepackage{tikz}

\theoremstyle{plain}
\newtheorem{theorem}{Theorem}

\newtheorem{proposition}[theorem]{Proposition}
\theoremstyle{definition}

\newcommand{\E}{{\mathbb{E}}}

\newcommand{\Var}{\mathrm{Var}}

\newcommand{\1}{\mathbbm{1}}

\let\oldmarginpar\marginpar
\renewcommand\marginpar[1]{\-\oldmarginpar[\raggedleft\footnotesize #1]%
{\raggedright\footnotesize #1}}

\definecolor{lightyellow}{RGB}{255,255,102}

\begin{document}

\title[]%
{The Bruss-Robertson Inequality: \\ Elaborations, Extensions, and Applications}
\author[]
{J. Michael Steele}

\thanks{J. M.
Steele:  Department of Statistics, The Wharton School,
University of Pennsylvania, 3730 Walnut Street, Philadelphia, PA, 19104.
Email address: \texttt{steele@wharton.upenn.edu}}

\citationmode{full}

\begin{abstract} The Bruss-Robertson inequality gives a bound on the
maximal number of elements of a random sample whose sum is less than a specified
value, and the extension of that inequality which is given here
neither requires the independence of the summands nor requires the equality of their marginal distributions.
A review is also given of the applications of the Bruss-Robertson inequality,
especially the applications to problems of combinatorial optimization such as the sequential
knapsack problem and the sequential monotone subsequence selection problem.

        \bigskip

        \noindent {\sc Key Words.} Order statistical inequalities, sequential knapsack problem, sequential monotone subsequence problem,
        sequential selection, online selection, Markov decision problems, Bellman equation.

        \smallskip

        \noindent {\sc Mathematics Subject Classification (2010).}
        Primary: 60C05, 60G40, 90C40; Secondary:  60F99, 90C27, 90C39

\end{abstract}


\maketitle



\section{Bruss-Robertson Inequality}\label{se:BR Inequality}

Here, at first,  we consider a finite sequence of
non-negative independent random variables $X_i$, $i=1,2, \ldots, n$  with a common continuous distribution function $F$, and,
given a real value $s>0$,
we are primarily concerned with the random variable
\begin{equation}\label{eq:BR maximal funtion}
M_n^*(s)=\max \{ \,|A|: \,\sum_{i \in A} X_i \leq s \},
\end{equation}
where, as usual, we use $|A|$ to denote the cardinality of the subset of integers $A \subset [n]=\{1,2, \ldots, n\}$.
We call $M_n^*(s)$ the Bruss-Robertson maximal function, and, one should note that in terms of the traditional order statistics,
$$
X_{n,1} < X_{n,2} < \cdots <X_{n,n},
$$
one can also write $M_n^*(s) =\max \{k: X_{n,1} + X_{n,2} +\cdots +X_{n,k} \leq s \}$.

In \citeasnoun{BruRob:AAP1991} it was found
that the expectation of the maximal function
$M_n^*(s)$ has an elegant bound in terms of the distribution function $F$ and a natural threshold value $t(n,s)$
that one defines by the implicit relation
\begin{equation}\label{eq:threshold def}
n \int_0^{t(n,s)} x \, dF(x) = s.
\end{equation}
Specifically, one learns from \citeasnoun[p.~622]{BruRob:AAP1991} that
\begin{equation}\label{eq:BR inequality}
\E[M_n^*(s)]\leq n F(t(n,s)),
\end{equation}
and the main goal here is to explore this inequality with an eye toward its mastery, its extensions and its combinatorial applications.

To gain a quick appreciation of the potential of the bound \eqref{eq:BR inequality}, it is useful to take $F$ to be the uniform distribution on $[0,1]$. By \eqref{eq:threshold def}
we have $t(n,s)=(2s/n)^{1/2}$ provided that $(2s/n)^{1/2}\leq 1$,
so for $s=1$ we find from \eqref{eq:BR inequality} that
for uniformly distributed random variables one always has
\begin{equation}\label{eq:BR inequality unif}
\E[\max \{ |A|: \sum_{i \in A \subset [n]} X_i \leq 1 \}] \leq \sqrt{2n}.
\end{equation}
This tidy bound already points the way to some of the most informative
combinatorial applications of the Bruss-Robertson inequality \eqref{eq:BR inequality}.

The next section elaborates on the proof of the Bruss-Robertson maximal inequality, and
in Section \ref{se:BR Extension}
we then see how the argument of Section \ref{se:BR Inequality Proof} needs only minor modifications in order to provide
an inequality of unexpected generality. After illustrating this new inequality
with three examples in Section \ref{se:three examples},
we turn in Section \ref{se:LIS connections} to the combinatorial applications.
Finally, Section \ref{se:conclusion} recalls other applications of the Bruss-Robinson maximal function including
recent applications to the theory of resource
dependent branching processes and the mathematical models of societal organization.

\section{An Elaboration of the Original Proof}\label{se:BR Inequality Proof}

The original proof of the Bruss-Robertson inequality \eqref{eq:BR inequality} is not long or difficult, but by
a reformulation and elaboration of that proof one does gain some concrete benefits.
These benefits are explained in detail in the next section, so, for the moment, we just focus on the proof of \eqref{eq:BR inequality}.

First, by the continuity of the joint distribution of $(X_i: i \in [n])$,
one finds that there is a unique set $A\subset [n]$ that attains the maximum
in the definition \eqref{eq:BR maximal funtion} of $M_n^*(s)$.
We denote this subset by $A(n,s)$, and we also
introduce a second set $B(n,s)\subset [n]$ that we define by setting
\begin{equation}\label{eq: B def}
B(n,s)= \{i: X_i \leq t(n,s) \},
\end{equation}
where $t(n,s)$ is the threshold value determined by the implicit relation \eqref{eq:threshold def}.

The idea behind the proof of the maximal inequality \eqref{eq:BR inequality} is to compare the sets $A(n,s)$ and $B(n,s)$,
together with their associated sums,
\begin{equation}\label{eq:two sums}
S_{A(n,s)}=\sum_{i \in A(n,s)} X_i \quad \text{and} \quad S_{B(n,s)}=\sum_{i \in B(n,s)} X_i.
\end{equation}
Here it is useful to note that by the definitions of these sums one has the immediate relations
\begin{equation}\label{eq:two relations}
S_{A(n,s)}\leq s \quad \text{and} \quad \E[S_{B(n,s)}]=n \int_0^{t(n,s)} x \, dF(x) = s.
\end{equation}

Now, by its definition, $S_{A(n,s)}$  is a partial sums of order statistics, and, since the summands of
$S_{B(n,s)}$ consists precisely of the values $X_{n,i}$ with $X_{n,i} \leq t(n,s)$, we see that
$S_{B(n,s)}$ is also equal to a partial sum of order statistics, even though the partial sums of $S_{B(n,s)}$ are not typically
partial sums of the order statistics.
These observations will help us with  estimations that depend on the relative sizes of the sum $S_{A(n,s)}$  and $S_{B(n,s)}$.

For example, if
$S_{B(n,s)}\leq S_{A(n,s)}$ then one has  $B(n,s) \subset A(n,s)$. Moreover, the summands $X_i$ with $i \in A(n,s) \setminus B(n,s)$ are all bounded
below by $t(n,s)$, so we have the bound
$$
S_{B(n,s)} + t(n,s)\{|A(n,s)|-|B(n,s)|\} \leq S_{A(n,s)} \quad \text{if } S_{B(n,s)}\leq S_{A(n,s)}.
$$
Similarly, if
$S_{A(n,s)} \leq S_{B(n,s)}$ then $A(n,s) \subset B(n,s)$ and the summands $X_i$ with $i \in B(n,s) \setminus A(n,s)$ are all bounded
above by $t(n,s)$; so, in this case,  we have the bound
$$
S_{B(n,s)}  \leq S_{A(n,s)} + t(n,s)\{|B(n,s)|-|A(n,s)|\} \quad \text{if } S_{A(n,s)}\leq S_{B(n,s)}.
$$
Taken together, the last two relations tell us that whatever the relative sizes of $S_{A(n,s)}$ and $S_{B(n,s)}$ may be,
one always has the key relation
\begin{equation}\label{eq:Key RV inequality}
t(n,s)\{|A(n,s)|-|B(n,s)|\} \leq S_{A(n,s)}- S_{B(n,s)}.
\end{equation}
Here $t(n,s)>0$ is a constant, $|A(n,s)|=M_n^*(s)$, and by \eqref{eq:two relations} the righthand side has non-positive expectation, so
taking the expectations in \eqref{eq:Key RV inequality} gives us
$$
\E[M_n^*(s)]\leq \E[|B(n,s)|]=\E[\sum_{i=1}^n \1(X_i \leq t(n,s))]=nF(t(n,s)),
$$
and the proof of the Bruss-Robertson inequality \eqref{eq:BR inequality} is complete.

\section{Extension of the Bruss-Robertson Inequality}\label{se:BR Extension}

The preceding argument has been organized so that it may be easily modified to give a
bound that is notably more general. Specifically, one does not need independence
for the Bruss-Robertson inequality \eqref{eq:BR inequality}. Moreover, after an appropriate modification of the definition of $t(n,s)$
one does not need to require that the observations have a common distribution.

\begin{theorem}[Extended Bruss-Robertson Inequality]\label{th:EBR Inequality}
If for each $i \in [n]$ the non-negative random variable $X_i$ has the continuous distribution
$F_i$, and if one defines $t(n,s)$ by the implicit relation
\begin{equation}\label{eq:new def of t(n,s)}
s=\sum_{i=1}^n \int_0^{t(n,s)} x \, dF_i(x),
\end{equation}
then one has
\begin{equation}\label{eq:RB inequality generalized}
\E[\max \big\{ |A|: \sum_{i \in A \subset [n]} X_i \leq s \big\}] \leq \sum_{i=1}^n F_i(t(n,s)).
\end{equation}
\end{theorem}

When the random variables $X_i$, $i \in [n]$, have a common distribution, then the defining condition \eqref{eq:new def of t(n,s)}
for $t(n,s)$ just recaptures the classical definition \eqref{eq:threshold def} of the traditional threshold value. In the same way,
the upper bound in
\eqref{eq:RB inequality generalized} also recaptures the upper bound of the original Bruss-Robertson inequality \eqref{eq:BR inequality}.

The proof of Theorem \ref{th:EBR Inequality} requires only some light modifications of the argument of Section \ref{se:BR Inequality Proof}.
Just as before, one defines $B(n,s)$ by \eqref{eq: B def},
but now some additional care is needed with the definition of  $A(n,s)$.

To keep
as close as possible to the argument of Section \ref{se:BR Inequality Proof},
we first define a total order on the set $\{X_i: i \in [n]\}$ by writing $X_i \prec  X_j$ if either one has $X_i < X_j$, or if
one has both $X_i=X_j$ and $i < j$.
Using this order, there is now unique permutation $\pi:[n] \rightarrow [n]$ such that
$$
X_{\pi(1)} \prec X_{\pi(2)} \prec \cdots \prec X_{\pi(n)},
$$
and one can then take $A(n,s)$
to be largest set $A \subset [n]$ of the form
\begin{equation}\label{eq:second def of A(n,s)}
A= \{ \pi(i): X_{\pi(1)} + X_{\pi(2)} + \cdots + X_{\pi(k)} \leq s \}.
\end{equation}

Given these modifications, one can then proceed with the proof of key inequality \eqref{eq:Key RV inequality}
essentially without change. We use the same definitions \eqref{eq:two sums} for the sums $S_{A(n,s)}$ and $S_{B(n,s)}$, so
by the new definition \eqref{eq:new def of t(n,s)} of $t(n,s)$, one now has
\begin{equation*}\label{eq:E of S sub B}
\E[S_{B(n,s)}]=\sum_{i=1}^n \int_0^{t(n,s)} x \, dF_i(x) = s.
\end{equation*}
Since we still have $S_{A(n,s)}\leq s$, the expectation on the right side of  \eqref{eq:Key RV inequality} is non-positive, and one
can complete
the proof of Theorem \ref{th:EBR Inequality} just as one completed the proof \eqref{eq:BR inequality} in Section \ref{se:BR Inequality Proof}.

At first it may seem surprising that one does not need independence in this
theorem, but in short order this becomes
self-evident. As the organization of
Section \ref{se:BR Inequality Proof} makes explicit, none of the required calculations
depend on the joint distribution of $(X_i: i \in[n] )$. More specifically, one just needs to note that
the argument of Section \ref{se:BR Inequality Proof}
depends exclusively on pointwise bounds and the linearity of expectation.

\section{Three Illustrative Examples}\label{se:three examples}

There are times when it is difficult to solve the non-linear relation \eqref{eq:new def of t(n,s)} for $t(n,s)$, but there
are also informative situations where this does not pose a problem, such as the three examples of this section. The first example
shows that one can deal quite easily with uniformly distributed random variables on multiple scales. The other two examples show
that when one considers dependent random variables, there are curious new phenomena that can arise.

\medskip
\noindent
{\sc Example 1. Basic Benefits}
\smallskip

Here,  for each $i \in [n]$ we take $X_i$ to be uniformly distributed on the real interval $[0,i]$, but we do not require that these
random variables to be independent.
If we also take $0 < s \leq 1$ and take $n \geq 4$ (for later convenience) then the defining condition \eqref{eq:new def of t(n,s)} tells us
$$
s=\frac{1}{2} \sum_{i=1}^n \frac{1}{i} t^2(n,s)=\frac{1}{2} t^2(n,s) H_n \quad \text{or} \quad  t(n,s)=\sqrt{2s/H_n},
$$
where as usual $H_n$ denotes the $n$'th harmonic number.
In particular, for $s=1$  the bound \eqref{eq:RB inequality generalized}
tells us that
\begin{equation*}
\E[\max \big\{ |A|: \sum_{i \in A} X_i \leq 1 \big\}] \leq \sum_{i=1}^n F_i(t(n,s))=\sum_{i=1}^n \frac{1}{i}({2}/{H_n})^{1/2}
=(2 H_n)^{1/2},
\end{equation*}
where we use $n\geq 4$ to assure that $H_n> 2$ and thus to keep the computation as simple as possible.
This bound offers an informative complement to \eqref{eq:BR inequality unif}, and, here again, one may underscore that no independence is required
for this inequality. The bound depends only on the marginal distribution of the $X_i$, $i \in [n]$.

\pagebreak
\medskip
\noindent
{\sc Example 2. Extreme Dependence}
\smallskip

Here we take $X$ to have the uniform distribution on $[0,1]$, and we set $X_i=X$ for all $i \in [n]$.
For specificity, we take $s=1$, and we find
from \eqref{eq:new def of t(n,s)} we find that $t(n,s)=(2s/n)^{1/2}$. Thus, just as one had for a sample of $n$ independent, uniformly distributed
random variables, the upper bound provided by \eqref{eq:RB inequality generalized} is given by $(2n)^{1/2}$.

Nevertheless, in this case the bound is not at all sharp. To see how poorly it does, we first note that
\begin{equation}\label{eq: example value of M}
M_n^*(1)=\max\{|A| : \sum_{i\in A} X_i \leq 1|]= \min \{n, \lfloor 1/X \rfloor \}.
\end{equation}
To evaluate the expectation of $M_n^*(1)$, we first recall that there is a useful variation of the usual formula for Euler's constant $\gamma=0.5772\ldots$ which was discovered by \citeasnoun{Polya1917} and
which tells us that
$$
\int_0^1 \left\{ \frac{1}{x} - \bigg\lfloor \frac{1}{x} \bigg\rfloor \right\} \, dx =1 -\gamma.
$$
Consequently, if we write the domain of integration as
$[0,1/n] \cup [1/n, 1]$ and note  that the integrand is bounded by $1$, then we have
$$
\int_0^1 \left\{
\min(n, \frac{1}{x})
- \min( n, \lfloor \frac{1}{x})  \rfloor )
\right\}
\, dx =1 -\gamma +O(\frac{1}{n}).
$$
The integral of the first term equals $1+\log n$, so upon returning to
\eqref{eq: example value of M} one finds
\begin{equation}\label{eq:ex2 formula}
\E[M_n^*(1)] =\E[\min \{n, \lfloor 1/X \rfloor \}]= \log n +\gamma - O(\frac{1}{n}).
\end{equation}
When we compare this to the
$(2n)^{1/2}$ bound that we get
from \eqref{eq:RB inequality generalized}, we see that it
falls uncomfortably far from the actual value  of $\E[M_n^*(1)]$. This illustrates in a simple way that
there is a price to be paid for the generality of Theorem \ref{th:EBR Inequality}.

One could have come to a similar conclusion with estimates that are less precise than \eqref{eq:ex2 formula}.
Nevertheless, there is some independent benefit to seeing Euler's constant emerge from the knapsack problem.
More critically, this example illustrates the reason for the more refined definition of $A(n,s)$ that was
introduced in \eqref{eq:second def of A(n,s)}. Here the
maximum in \eqref{eq: example value of M} is typically attained for many different choices of $A \subset [n]$.
Nevertheless, with help from the total order $\prec$ one regains uniqueness in definition of $A(s,n)$,
and, as a consequence, the logic of Section \ref{se:BR Inequality} serves one just as well as it did before.

\medskip
\noindent
{\sc Example 3. Beta Densities and a Long Monotone Sequence}
\smallskip

Now, for each $i \in [n]$ we take $X_i$ to have the $\texttt{Beta}(i, n-i+1)$ density, so in particular, $X_i$ has the
same marginal distribution as the $i$'th smallest value $U_{(i)}$ in a sample $\{U_1, U_2, \ldots, U_n\}$ of $n$ independent
random variables with the uniform distribution on $[0,1]$. Still, for the moment we make no assumption about the joint
distribution of $(X_i: i \in [n])$.
By the condition
\eqref{eq:new def of t(n,s)} we then have
\begin{align*}
s&=\sum_{i=1}^n \int_0^{t(n,s)} x \, dF_i(x) =
\E\big[ \sum_{i=1}^n  U_{(i)} \1[ U_{(i)} \in [0,{t(n,s)}] \big]\\
&=\E\big[ \sum_{i=1}^n  U_{i} \1[ U_i \in [0,{t(n,s)}] \big]= \frac{1}{2} t^2(n,s) n,
\end{align*}
so in this case we again find $t(n,s)=(2s/n)^{1/2}$. Thus, by \eqref{eq:RB inequality generalized} we have
the upper bound $(2n)^{1/2}$ when $s=1$, and our bound echoes what we know from the classical inequality \eqref{eq:BR inequality unif}.

This inference depends only on our assumption about the marginal distributions, but
one can go a bit further if we assume the equality of the joint distributions $(X_i: i \in [n])$ and $(U_{(i)}: i \in [n])$. In particular,
one finds in this case that our upper bound $(2n)^{1/2}$ is essentially tight.

It is also evident in this case that one has
$X_1 < X_2 < \cdots < X_n$, and this observation
would be quite uninteresting, except that in the next section we will find in
that in the independent case there is a remarkable link between monotone subsequences and the Bruss-Robertson inequality.
Thus, it is something of a curiosity to see how thoroughly this connection can be broken while still retaining the bound given by
the general
inequality of Theorem \ref{th:EBR Inequality}.

\section{Sequential Subsequence Selection Problems}\label{se:LIS connections}

A basic source of interest in the Bruss-Robertson inequality \eqref{eq:BR inequality} and its generalization \eqref{eq:RB inequality generalized}
is that these results lead to \emph{a priori} upper bounds for two well studied problems in combinatorial optimization.
In particular, in the classical case of independent uniformly distributed random variables,
the Bruss-Robertson inequality \eqref{eq:BR inequality} gives bounds that are that are essentially sharp
for both the sequential knapsack problem and the sequential increasing subsequence selection problem.

In the sequential knapsack problem, one observes a sequence of $n$ independent non-negative random variables $X_1, X_2, \ldots , X_n$
with a fixed, known distribution $F$. One is also given real value $x \in [0,\infty)$ that one regards as the capacity of a knapsack into which
selected items are placed.  The
observations are observed sequentially, and, at time $i$, when $X_i$ is first observed, one either selects $X_i$ for inclusion in the knapsack or
else $X_i$  is rejected from any future consideration. The goal is to maximize the expected
\emph{number} of items that are included in the knapsack. Since the Bruss-Robertson maximal function \eqref{eq:BR maximal funtion} tells one
how well one could if one knew in advance all of the values $\{X_i: i \in [n]\}$, it is
evident that no strategy for making sequential choices can ever lead to more affirmative choices than $M_n(x)$.

The sequential knapsack problem is a Markov decision problem that is known to have an optimal sequential selection strategy that is given
by a unique non-randomized Markovian decision rule. When one follows this optimal policy beginning with $n$ values to be observed and
with an initial knapsack capacity of $x$, the expected number of selections that one makes is denoted by $v_n(x)$. This is called the value function for the Markov decision problem, and, it can be calculated by the recursion relation
\begin{equation}\label{eq:Bellman Knapsack}
v_n(x)=(1-F(x)) v_{n-1}(x) + \int_0^x \max \{ v_{n-1}(x), 1+ v_{n-1}(x-y) \}\, dF(y).
\end{equation}
Specifically, one begins with the obvious relation $v_0(x) \equiv 0$, and one computes $v_n(x)$ by iteration of \eqref{eq:Bellman Knapsack}.

This is called the Bellman equation (or optimality equation)
for the sequential knapsack problem, and it is easy to justify. The first term comes from the possibility that
$X_1$ is too large to fit into the knapsack, and this event happens with probability $1-F(x)$. In this case, one cannot accept $X_1$, so one
is left with the original capacity $x$ and there are only $n-1$ more values to be observed. This gives one the first term of
\eqref{eq:Bellman Knapsack}.

For the more interesting second
term of \eqref{eq:Bellman Knapsack}, we consider the case where one has
$X_1=y \leq x$, so one has the option either to accept or to reject $X_1$. If we reject $X_1$, we have no increment to our knapsack count
and we have the value $v_{n-1}(x)$ for the expected number of selections from the remaining values. On the other hand, if we accept $X_1$,
we have added $1$ to our knapsack count. We also have a remaining capacity of $x-y$, and we have $n-1$ observations to be seen. One takes the
best of these two values, and this gives us the second term of \eqref{eq:Bellman Knapsack}.

Now we consider the problem of sequential selection of a monotone decreasing
subsequence. Specifically, we observe sequentially $n$ independent
random variables $X_1, X_2, \ldots, X_n$ with the common continuous distribution $F$, and we make monotonically decreasing choices
$$
X_{i_1} > X_{i_2} > \cdots > X_{i_k}.
$$
Our goal here is to maximize the expected
number of choices that we make. Again we have a Markov decision problem with an unique optimal non-randomized Markov decision policy.
Here, prior to making any selection, we take the state variable $x$ to be the supremum of the support of $F$, which may be infinity. After
we have made at least one selection, we take the state variable $x$ to be the value of the last selection that was made.

Now we write $\widetilde{v}_n(x)$ for the expected number of selections made under the optimal policy
when the state variable is $x$ and where there are $n$ observations that remain to be observed.
In this case the Bellman equation given by
\citeasnoun{SamSte:AP1981} can be written as
\begin{equation}\label{eq:Bellman SLDS}
\widetilde{v}_n(x)=(1-F(x)) \widetilde{v}_{n-1}(x) + \int_0^x \max \{ \widetilde{v}_{n-1}(x), 1+ \widetilde{v}_{n-1}(y) \}\, dF(y),
\end{equation}
where again one has the obvious relation $\widetilde{v}_n(x) \equiv 0$ for the initial value.
In \eqref{eq:Bellman SLDS}
the decision to select $X_1=y$ would move the state variable to $y$, so here we have
$1+\widetilde{v}_{n-1}(y)$ where earlier we had the term  $1+ v_{n-1}(x-y)$ in the knapsack Bellman equation \eqref{eq:Bellman Knapsack}.
In knapsack problem the state variable moves from $x$ to $x-y$ when $X_1=y$ is selected.

In general, the solutions of \eqref{eq:Bellman Knapsack} and
\eqref{eq:Bellman SLDS} are distinct. Nevertheless, \citeasnoun{CofFlaWeb:AAP1987} observed that $v_n(x)$
and $\widetilde{v}_n(x)$ are equal when the observations are uniformly distributed.
This can be proved formally by an inductive argument that uses the two Bellman equations
\eqref{eq:Bellman Knapsack} and \eqref{eq:Bellman SLDS}.

\begin{proposition} \label{pr: equal vs}
If $F(x)=x$ for $0 \leq x \leq 1$, then one has
$$
v_n(x) =\widetilde{v}_n(x) \quad \text{for all } n \geq 0 \, \text{and } 0 \leq x \leq 1.
$$
\end{proposition}
\begin{proof}
For $n=0$ we have $v_0(x) =\widetilde{v}_0(x)=0$ for all $x \in [0,1]$, and this gives us the base case for an induction.
To make the inductive step from $n-1$ to $n$, we first use the Bellman equation \eqref{eq:Bellman Knapsack}
and then use the  induction hypothesis to get
\begin{align*}
v_n(x)&=(1-x) v_{n-1}(x) + \int_0^x \max \{ v_{n-1}(x), 1+ v_{n-1}(x-y) \}\, dy \\
& = (1-x) \widetilde{v}_{n-1}(x) + \int_0^x \max \{ \widetilde{v}_{n-1}(x), 1+ \widetilde{v}_{n-1}(x-y) \}\, dy \\
& = (1-x) \widetilde{v}_{n-1}(x) +\int_0^x \max \{ \widetilde{v}_{n-1}(x), 1+ \widetilde{v}_{n-1}(y) \}\, dy = \widetilde{v}_{n}(x),
\end{align*}
where in passing to the last line one uses the symmetry of the uniform measure on $[0,x] \subset [0,1]$.
Naturally, for the last equality one just needs to use the second Bellman equation \eqref{eq:Bellman SLDS}.
\end{proof}

Despite the equality of the value functions established by this proposition, no one has yet found any
direct choice-by-choice coupling between the sequential knapsack problem and the sequential monotone subsequence
selection problem.
Nevertheless, one can create a detailed linkage between these two problems that does yield more
than just the equality of the associated expected values.

The first step is to note that the equality of the value functions
permits one to construct optimal selection rules that can be applied simultaneously to the same sequence of
observations. The selections that are made will be different in the two problems, but one still finds useful
distributional relationships.

\medskip
\noindent
{\sc Threshold Strategies from Value Functions}
\smallskip

The essential observation is that the second term of the
Bellman equation \eqref{eq:Bellman SLDS} leads one almost immediately to the construction of an optimal
selection strategy for the monotone subsequence problem. These strategies lead one in turn to a more
detailed understanding of number of values that one actually selects.

First, one notes that it is easy to show (cf.~\citeasnoun{SamSte:AP1981}) that
there is a unique $y \in [0,1]$ that solves the ``equation of indifference":
$$
\widetilde{v}_{n-1}(x)= 1+ \widetilde{v}_{n-1}(y).
$$
We denote this solution by $\alpha_n(x)$, and we use its values to determine the rule for making the sequential selections.

At the moment just before
$X_i$ is presented, we face the problem of selecting a monotone sequence from among the $n-i+1$ values $X_i, X_{i+1}, \ldots, X_n$,
and if we let $\widetilde{S}_{i-1}$ denote the last of the values $X_1, X_2, \ldots, X_{i-1}$ that has been selected so far,
then we can only select
$X_i$ if it is not greater than the most recently selected value $\widetilde{S}_{i-1}$.
In fact, one would choose to select $X_i$ if and only if it falls in the interval
$[S_{i-1}, S_{i-1}-\alpha_{n-i+1}(\widetilde{S}_{i-1})]$.
Thus, the actual
number of values selected out of the original $n$ is the random variable given by
\begin{equation*}\label{eq:monotone achieved}
\widetilde{V}_n \stackrel{\text{def}}{=} \sum_{i=1}^n \1 (X_i \in [\widetilde{S}_{i-1},\, \widetilde{S}_{i-1}-\alpha_{n-i+1}(\widetilde{S}_{i-1})]).
\end{equation*}
By the same logic, one finds that in the sequential knapsack problem the number of values
that are selected is by the optimal selection rule can be written as
\begin{equation*}\label{eq:knapsack achieved}
{V}_n \stackrel{\text{def}}{=} \sum_{i=1}^n \1 (X_i \in [0, \alpha_{n-i+1}(S_{i-1})]),
\end{equation*}
where now $S_{i-1}$ denotes the capacity that remains after all of the knapsack selections have been made
from the set of values  $X_1, X_2, \ldots, X_{i-1}$ that have already been observed.

By this parallel construction and by Proposition \ref{pr: equal vs}, we have
$$\E[{V}_n]= v_n(1)=\widetilde{v}_{n}(1)=\E[\widetilde{V}_n],$$
but considerably more is true. Initially, one has $S_0=1= \widetilde{S}_0$,
so, one further finds the equality of the joint distributions of the vectors
$$
(S_0, S_1, \ldots, S_{n-1}) \quad \text{and} \quad (\widetilde{S}_0, \widetilde{S}_1, \ldots, \widetilde{S}_{n-1}),
$$
since the two processes $\{S_i: 0\leq i \leq n\}$ and $\{\widetilde{S}_i: 0\leq i \leq n\}$  are (temporally non-homonomous) Markov chains
that have the same transition kernel at each time epoch.

This equivalence tells us in turn that the partial sums
\begin{align*}
{V}_{n,k} & \stackrel{\text{def}}{=}
\sum_{i=1}^k \1 (X_i \in [0, \alpha_{n-i+1}(S_{i-1})]) \quad \text{and}\\
\widetilde{V}_{n,k} & \stackrel{\text{def}}{=}
\sum_{i=1}^k \1 (X_i \in [\widetilde{S}_{i-1},\, \widetilde{S}_{i-1}-\alpha_{n-i+1}(\widetilde{S}_{i-1})]),
\end{align*}
satisfy the distributional identity
\begin{equation}\label{eq:equivalence}
{V}_{n,k} \stackrel{\text{d}}{=} \widetilde{V}_{n,k}  \quad \text{for all } 1 \leq k \leq n.
\end{equation}
Nevertheless, the two processes
$\{ {V}_{n,k}: 1 \leq k \leq n \}$ and $\{\widetilde{V}_{n,k}: 1 \leq k \leq n \}$
are not equivalent as processes. Despite the equality of the marginal distributions, the joint distributions are wildly different.

\medskip
\noindent
{\sc Classical Consequences}
\smallskip

The Bruss-Robertson inequality \eqref{eq:BR inequality unif} tells us directly that $\E[{V}_n ] \leq \sqrt{2n}$, so,
by the distributional identity  \eqref{eq:equivalence}, we find indirectly that
\begin{equation}\label{eq:BRagain}
\E[\widetilde{V}_n ]  \leq \sqrt{2n} \quad \text{for all } n \geq 1.
\end{equation}
It turns out that \eqref{eq:BRagain} can be proved by a remarkable variety of methods. In particular,
\citeasnoun{Gne:JAP1999} gave a direct proof \eqref{eq:BRagain} where one can even accommodate a random sample size $N$
and where the upper bound of \eqref{eq:BRagain} is replaced with the natural proxy $(2 \E[N])^{1/2}$.
More recently, \citeasnoun{ArlottoMosselSteele:2015} gave two further proofs of \eqref{eq:BRagain} as consequences of
bounds that were developed for the \emph{quickest selection problem}, a sequential decision problem that
provides a kind of combinatorial dual
to the classical sequential selection problem.

The distributional identity \eqref{eq:equivalence} can also be used to make some notable inferences about the knapsack problem
from what has been discovered in the theory of sequential monotone selections. For example, by building on
the work of
\citeasnoun{BruDel:SPA2001} and \citeasnoun{BruDel:SPA2004}, it was found in
\citeasnoun{ArlNguSte:SPA2015} that one has
\begin{equation} \label{eq: Var and CLT for SMSP}
\Var[\widetilde{V}_n] \sim \frac{1}{3} \sqrt{2n} \quad \text{and } \quad
\frac
{\widetilde{V}_n -\sqrt{2n}}
{3^{-1/2} (2n)^{1/4}}
\Rightarrow {\rm N}(0,1).
\end{equation}
Thus, as a consequence of the distributional identity \eqref{eq:equivalence} one
has the same results for the knapsack variable $V_n$ for independent observations
with the uniform distribution on $[0,1]$.

It seems quite reasonable to conjecture that there are results that are
analogous to \eqref{eq: Var and CLT for SMSP} for the sequential knapsack problem
where the driving distribution $F$ is something other than the uniform. Nevertheless,
the proof of any such results is unlikely to be easy since the proofs of the
relations of \eqref{eq: Var and CLT for SMSP} required a sustained
investigation of the Bellman equation \eqref{eq:Bellman SLDS}.

Finally, one should also recall the \emph{non-sequential} (or clairvoyant) selection problem where
one studies the random variable $$L_n=\max\{k: X_{i_1}< X_{i_2}< \ldots < X_{i_k}, \, \, 1\leq i_1< i_2< \cdots <i_k \leq n \}.$$
This classic problem has a long history that is beautifully told in \citeasnoun{Rom:CUP2014}. Here the most relevant part of that story is that
\citeasnoun{BaikDeiftKurt99} found the asymptotic distribution of $L_n$, and, in particular, they found that one has
the asymptotic relation
\begin{equation}\label{eq:BDK99}
\E[L_n]= 2 \sqrt{n} -\alpha n^{1/6} + o(n^{1/6}) \quad \text{where } \alpha=1.77108...
\end{equation}

Ironically, it is still not known if the map $n \mapsto \E[L_n]$ is concave, even though this seems like an
exceptionally compelling conjecture.
The estimate \eqref{eq:BDK99} certainly suggest this, and, moreover,
we already know from \citeasnoun[p.~3604]{ArlNguSte:SPA2015} that for the analogous sequential problems
the corresponding map
$n \mapsto \E[\widetilde{V}_n ]=\E[{V}_n ]$
is indeed concave.

\section{Further Connections and Applications}\label{se:conclusion}

Here the whole goal has been to explain, extend, and explore the
upper bound given by the Bruss-Robertson inequality. Still, there is second side to
the Bruss-Robertson maximal function, and both lower bounds and asymptotic relations have been developed in
investigations by
\citeasnoun{CofFlaWeb:AAP1987},
\citeasnoun{BruRob:AAP1991}, and
\citeasnoun{RheTal:JAP1991}.
Furthermore,
\citeasnoun{BoshuizenKertzAAP1999} have even established the joint
convergence in distribution of the (suitably normalized) Bruss-Robertson maximal function and
a sequence of approximate solutions to the sequential knapsack problem, although this result does
not seem to speak directly to the problem of proving a wider analog of \eqref{eq: Var and CLT for SMSP}.

The applications that have been considered here were all taken
from combinatorial optimization. Nevertheless, there are several other
areas where the Bruss-Robertson maximal function \eqref{eq:BR maximal funtion} has a natural place. For example,
\citeasnoun{Gribonval-et-al-IEEE2012} use bounds from \citeasnoun{BruRob:AAP1991} in their study of
compressibility of high dimensional distributions.

Finally, one should note that the bounds of \citeasnoun{BruRob:AAP1991}
have a natural role in the theory
of resource dependent branching processes, or RDBPs. This is a rich theory
that in turn takes a special place in the recent work of \citeasnoun{BrussDuerinckxAAP2015} on
new mathematical models of
societal organization. These models have been further explained  for a broader (but still mathematical) audience in \citeasnoun{BrussJDMV2014},
where the theory of \citeasnoun{BrussDuerinckxAAP2015} is also applied to two contemporary European public policy issues.

\section*{Acknowledgements}

The author is pleased to thank Alessandro Arlotto and Peichao Peng for their useful comments on earlier incarnations of this
investigation.

\end{document}